\documentclass[12 pt, reqno]{amsart}

\usepackage{amsthm,amssymb,amstext,amscd,amsfonts,amsbsy,amsxtra,latexsym,cite}
\usepackage{mathrsfs}
\usepackage{verbatim}
\usepackage[latin1]{inputenc}
\usepackage{graphicx, graphics, wrapfig}

\newcommand{\field}[1]{\mathbb{#1}}

\newcommand{\Z}{\field{Z}}
\newcommand{\R}{\field{R}}

\def\({\left(}
\def\){\right)}

\DeclareMathOperator{\Li}{Li}
\newcommand{\coeff}{\operatorname{coeff}}

\theoremstyle{plain}
\newtheorem{theorem}{Theorem}
\newtheorem*{theorem*}{Theorem}
\newtheorem{lemma}[theorem]{Lemma}
\newtheorem{proposition}[theorem]{Proposition}

\newtheorem*{conjecture*}{Conjecture}

\theoremstyle{definition}

\theoremstyle{remark}

\newtheorem*{remark*}{Remark}

\numberwithin{theorem}{section} \numberwithin{equation}{section}

\begin{document}

\title{Asymptotic formulas for stacks and unimodal sequences}
\author{Kathrin Bringmann}
\address{Mathematical Institute\\University of
Cologne\\ Weyertal 86-90 \\ 50931 Cologne \\Germany}
\email{kbringma@math.uni-koeln.de}
\author{Karl Mahlburg}
\address{Department of Mathematics \\
Louisiana State University \\
Baton Rouge, LA 70802\\ U.S.A.}
\email{mahlburg@math.lsu.edu}

\subjclass[2000] {05A15, 05A16, 05A17, 11P81, 11P82}

\date{\today}

\keywords{unimodal sequences; generating functions; asymptotic formulas; integer partitions; Tauberian theorems}

\thanks{The research of the first author was supported by the Alfried Krupp Prize for
Young University Teachers of the Krupp Foundation.  The second author was supported by NSF Grant DMS-1201435.}

\begin{abstract}
In this paper, we  study  enumeration functions for unimodal sequences of positive integers, where the size of a sequence is the sum of its terms.  We survey known results for a number of natural variants of unimodal sequences, including Auluck's generalized Ferrer diagrams, Wright's stacks, and Andrews' convex compositions.  These results describe combinatorial properties, generating functions, and asymptotic formulas for the enumeration functions.  We also prove several new asymptotic results that fill in the notable missing cases from the literature, including an open problem in statistical mechanics due to Temperley.  Furthermore, we explain the combinatorial and asymptotic relationship between partitions, Andrews' Frobenius symbols, and stacks with summits.
\end{abstract}

\maketitle

\section{Introduction and statement of results}

One of the more ubiquitous concepts in enumerative combinatorics is unimodality, as a large number of common combinatorial functions have this property (see Stanley's survey articles \cite{Stan80,Stan89} for many examples and applications).  In particular, a {\it unimodal sequence of size $n$} is a sequence of positive integers that sum to $n$ such that the terms are monotonically increasing until a peak is reached, after which the terms are monotonically decreasing.  This means that the sequence can be written in terms of parts
\begin{equation*}
a_1, a_2, \dots, a_r, c, b_s, b_{s-1}, \dots, b_1,
\end{equation*}
that satisfy the inequalities
\begin{align}
\label{E:ac}
& 1 \leq a_1 \leq a_2 \leq \dots \leq a_r \leq c, \qquad \text{and} \\
\label{E:cb}
& c \geq b_s \geq b_{s-1} \geq \dots \geq b_1 \geq 1.
\end{align}
When writing examples of unimodal sequences, we  typically write the sequences without commas in order to save space. For example, $1244322$ is a unimodal sequence of size $18$, and all of the unimodal sequences of size $4$ are concisely listed as follows:
\begin{equation*}
4; \; 31; \; 13; \; 22; \; 211; \; 121; \; 112; \; 1111.
\end{equation*}

This simple combinatorial definition and its variants have also appeared in the literature under many other names and contexts, including ``stacks'' \cite{Aul51, Wri68, Wri71, Wri72}, ``convex compositions'' \cite{And12}, integer partitions \cite{And98, And11}, and in the study of ``saw-toothed'' crystal surfaces  in statistical mechanics \cite{Tem52}.  Furthermore, it is also appropriate to mention the closely related concept of ``concave compositions'', which were studied in \cite{And84, And11, And12, ARZ12, Rho12}, although we will not further discuss these sequences.

The purposes of this paper are twofold.  First, we concisely review previous work on enumeration functions related to stacks, unimodal sequences, and compositions, including the corresponding generating functions, and their asymptotic behavior when known.  Second, we provide new asymptotic formulas that successfully fill in the notable missing cases from the literature.  Along the way we describe some of the analytic techniques that were used in proving the previously known asymptotic formulas, including the theory of modular forms, Euler-MacLaurin summation, and Tauberian theorems; broadly speaking, these methods allow us to use the analytic properties of generating series in order to determine the asymptotic behavior of their coefficients.
In order to prove the new asymptotic formulas, we also introduce additional techniques such as the Constant Term Method and Saddle Point Method, which have been less widely used in this combinatorial setting.

Before presenting the definitions of the many variants of unimodal sequences, we first highlight a potential ambiguity in the enumeration of sequences that satisfy \eqref{E:ac} and \eqref{E:cb}.  In particular, if the largest part is not unique, then there are multiple choices for the ``summit'' $c$.  There are therefore two natural enumeration functions, as these choices may or may not be counted distinctly, and the majority of the enumeration functions thus occur in closely related pairs, corresponding to whether or not the summit is distinguished.  For the first type of enumeration function, we  canonically choose the summit $c$ to be the last instance of the largest part in the sequence.

We now describe the various enumeration functions for unimodal sequences, adopting Wright's terminology as much as possible -- we refer to all examples as ``stacks'' of various types, and we represent unimodal sequences visually by diagrams in which a sequence of blocks are stacked in columns.  The first definition is that of a {\it stack}, which is simply a unimodal sequence with $c > b_s$; this means that each distinct sequence is counted once, without regard to the multiplicity of the largest part.
Following Wright's notation in \cite{Wri68}, we denote the number of stacks of size $n$ by $s(n)$ (this function has also been called $Q(n)$ in the enumeration of ``Type B partitions'' found in \cite{Aul51}, and $u(n)$ in Section 2.5 of \cite{Stan11}).

Following the alternative method of enumeration, a {\it stack with summit} is defined to be a stack in which one of the largest parts is marked (this represents the possible different choices of $c$ in \eqref{E:ac} and \eqref{E:cb}).  In other words, we distinguish between stacks in which a different instance of the largest part is marked as the summit.  We write the enumeration function for stacks with summits of size $n$ as $ss(n)$; this is also found as $\sigma\sigma(n)$ in Section 3 of \cite{And84}, as $v(n)$ in the study of ``V-partitions'' in Section 2.5 of \cite{Stan11}, and as $X(n-1)$ in \cite{And12}, where Andrews' definition of ``convex compositions'' (which implicitly have a unique largest part) is equivalent to stacks with summits, up to a minor normalization.

For example, the stacks of size $4$ were listed previously,
and if we represent the marked parts as overlined,  then the stacks with summits of size $4$ are
\begin{equation*}
\overline{4}; \; \overline{3}1; \; 1\overline{3}; \; \overline{2}2; \; 2\overline{2}; \; \overline{2}11; \; 1\overline{2}1;
\; 11\overline{2}; \; \overline{1}111; \; 1\overline{1}11; \; 11\overline{1}1; \; 111\overline{1},
\end{equation*}
so $s(4) = 8$ and $ss(4) = 12.$  See Figures \ref{F:Stack} and \ref{F:StackSummit} for examples of the diagrams of stacks and stacks with summits.
The relevant generating functions, which first appeared in \cite{Aul51} and \cite{And84}, respectively, are
\begin{align}
\label{E:sgen}
\mathscr{S}(q) := \sum_{n \geq 0} s(n) q^n & = 1+\sum_{m \geq 1} \frac{q^m}{(q)_{m-1}^2 (1 - q^m)}, \\
\label{E:ssgen}
\mathscr{S}_s(q) := \sum_{n \geq 0} ss(n) q^n & = \sum_{m \geq 0} \frac{q^m}{(q)_{m}^2}.
\end{align}
Here we have used the standard notation for $q$-factorials, namely
$$
(a)_n = (a;q)_n := \prod_{j=0}^{n-1} \left(1 - aq^j\right)
$$
 for $n \geq 0$.    In order to understand these generating functions, observe that the $m$th term generates all stacks (respectively, stacks with summits) with largest part exactly $m$.

\begin{figure}[here]
\begin{center}
\includegraphics[width = 175 pt]{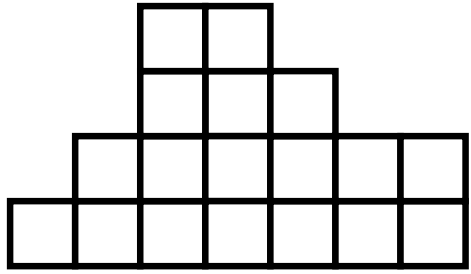}
\end{center}
\caption{\label{F:Stack}
Diagram for the stack $1244322$.}
\end{figure}

\begin{figure}[here]
\begin{center}
\includegraphics[width = 400 pt]{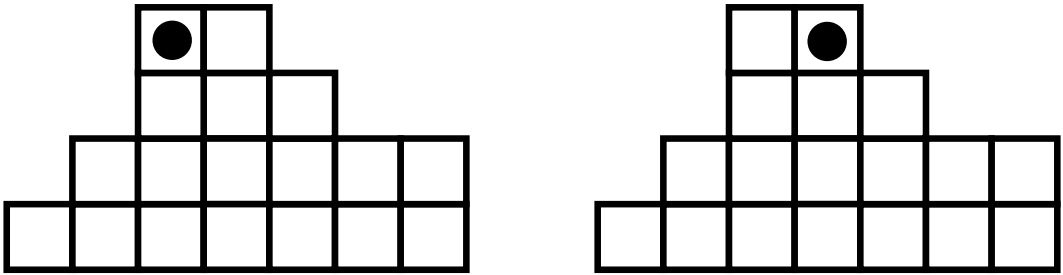}
\end{center}
\caption{\label{F:StackSummit}
Diagrams for the stacks with summits $12\overline{4}4322$ and $124\overline{4}322$, respectively, where the summits are distinguished by a mark at the top of the column.
}
\end{figure}

The next definition is that of a {\it receding stack}, which must satisfy \eqref{E:ac} and \eqref{E:cb} as well as the additional conditions
\begin{align}
\label{E:arec}
a_j & \geq a_{j+1} - 1 \quad \text{for } 1 \leq j \leq r-1, \\
\label{E:brec}
b_j & \leq b_{j+1} + 1 \quad \text{for } 1 \leq j \leq s-1.
\end{align}
In other words, the collections of $a_j$s and $b_j$s must both contain each part size from $1$ to $c-1$ at least once.  Also, as in the case of stacks above, we further require $c > b_s$ for this first type of enumeration function.  Again following Wright \cite{Wri71}, we denote the number of receding stacks of size $n$ by $g(n)$ (also written as $P(n)$ in the enumeration of ``Type A partitions'' in \cite{Aul51}).  Similarly, a {\it receding stack with summit} distinguishes one of the largest parts with a mark, and we denote the corresponding enumeration function by $gs(n)$.

For example, the receding stacks of size $4$ are
\begin{equation*}
121; \; 1111;
\end{equation*}
and the receding stacks with summits of size $4$ are
\begin{equation*}
1\overline{2}1;  \; \overline{1}111; \; 1\overline{1}11; \;
11\overline{1}1; \; 111\overline{1};
\end{equation*}
so $g(4) = 2$ and $gs(4) = 5.$
The generating functions for receding stacks (resp. with summits) are
\begin{align}
\label{E:ggen}
\mathscr{G}(q) := \sum_{n \geq 0} g(n) q^n & = 1 + \sum_{m \geq 1} \frac{q^{m^2}}{(q)_{m-1}^2 (1 - q^m)}, \\
\label{E:gsgen}
\mathscr{G}_s(q) := \sum_{n \geq 0} gs(n) q^n & = \sum_{m \geq 0} \frac{q^{m^2}}{(q)_{m}^2}.
\end{align}
The generating function for $g(n)$ was first given in \cite{Aul51}, but to our knowledge, receding stacks with summits have not previously been described.  This is likely due to the fact that \eqref{E:gsgen} is better known as the generating function for $p(n)$, Euler's integer partition function, which has been widely studied in its own right \cite{And98}.  We will describe the combinatorial relationship between partitions and receding stacks with summits in greater detail later in this paper.

Continuing with the definitions, a {\it shifted stack} is a sequence that satisfies \eqref{E:ac}, \eqref{E:cb}, and \eqref{E:arec}, with the additional condition $c > b_s$.  Although this definition might also naturally be called ``left-receding stacks'', the terminology is inspired by Auluck's  \cite{Aul51} and Wright's \cite{Wri71} observations that if each row is shifted left by half of a cell, then this is combinatorially equivalent to counting stacks of coins, or circular logs, with a single ``peak''.  See Figure \ref{F:Shifted} for an example of a shifted stack represented by both types of diagrams (cf. Figure 2 in Auluck's original work \cite{Aul51}).   The number of shifted stacks of size $n$ is denoted by $h(n)$ as in \cite{Wri71} (this function is also found as $R(n)$ in the enumeration of ``Type C partitions'' in \cite{Aul51}).  Similarly,  a {\it shifted stack with summit} is a shifted stack with a distinguished occurrence of the largest part, and the enumeration function is now denoted by $hs(n)$; these stacks previously appeared in Section 3 of \cite{And84}, where they were referred to as ``gradual stacks with summits'', with enumeration function $g\sigma(n)$.

\begin{figure}[here]
\begin{center}
\includegraphics[width = 300 pt]{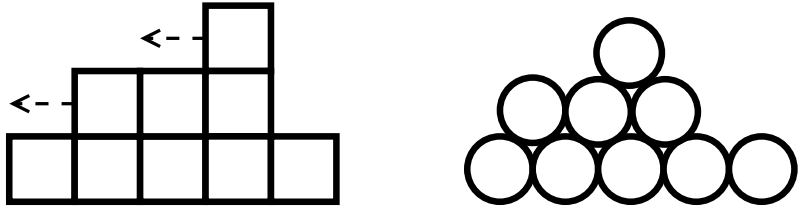}
\end{center}
\caption{\label{F:Shifted}
Two equivalent diagrams for the shifted stack $12231$.  The stack of coins diagram is obtained from the block diagram by successively shifting each row a half space to the left, and then replacing each block with a coin.}
\end{figure}

For example, the shifted stacks of size $4$ are
\begin{equation*}
121; \; 112; \; 1111;
\end{equation*}
and the shifted stacks with summits of size $4$ are
\begin{equation*}
1\overline{2}1; \; 11\overline{2}; \; \overline{1}111; \; 1\overline{1}11; \; 11\overline{1}1; \;
111\overline{1};
\end{equation*}
so $h(4) = 3$ and $hs(4) = 6.$
The generating functions for the two types of shifted stacks are first found in \cite{Aul51} and \cite{And84}, respectively, and are given by
\begin{align}
\label{E:hgen}
\mathscr{H}(q) := \sum_{n \geq 0} h(n) q^n & = 1 + \sum_{m \geq 1} \frac{q^{\frac{m(m+1)}{2}}}{(q)_{m-1}^2 (1 - q^m)}, \\
\label{E:hsgen}
\mathscr{H}_s(q) := \sum_{n \geq 0} hs(n) q^n & = \sum_{m \geq 0} \frac{q^{\frac{m(m+1)}{2}}}{(q)_{m}^2}.
\end{align}

The final variants that we discuss in this paper were introduced by Andrews in \cite{And12}, where he studied  ``strict convex compositions''.  We therefore define a {\it strict stack} to be a stack in which all of the inequalities in \eqref{E:ac} and \eqref{E:cb} are strict, and write the enumeration function as $d(n)$ (this appeared as $x_d(n)$ in \cite{And12}, and as the enumeration function $u^\ast(n)$ for ``strict unimodal sequences'' in \cite{Rho12}). Furthermore, a {\it semi-strict stack} (or ``left-strict'') requires strict inequalities only in \eqref{E:ac}, with the additional condition that the largest value is unique, so $c > b_s.$  We denote the enumeration function for semi-strict stacks by $dm(n)$ (this was denoted by $x_m(n)$ in \cite{And12}).   Note that the summits are already unique for strict and semi-strict stacks, so there are not enumeration functions of the second type.

For example, the strict stacks of size $4$ are
\begin{equation*}
4; \; 31; \; 13; \; 121;
\end{equation*}
and the semi-strict stacks of size $4$ are
\begin{equation*}
4; \; 31; \; 13; \; 211; \; 121;
\end{equation*}
so $d(4) = 4$ and $dm(4) = 5.$
The generating function for strict and semi-strict stacks are also found in \cite{And12}, and are given by
\begin{align}
\label{E:dgen}
\mathscr{D}(q) := \sum_{n \geq 0} d(n) q^n & = \sum_{m \geq 0} q^{m+1} \left(-q\right)_m^2, \\
\label{E:dmgen}
\mathscr{D}_m(q) := \sum_{n \geq 0} dm(n) q^n & = \sum_{m \geq 0} \frac{(-q)_m q^{m+1}}{(q)_m}.
\end{align}

All of the previously known asymptotic results for stacks are listed in Table \ref{Tab:stack}.  It is important to point out that in many cases the enumeration function is known to much greater precision than its main asymptotic term; this includes the full asymptotic expansion for $g(n)$ and $h(n)$ \cite{Wri71}, an exact formula for $gs(n)$ \cite{Rademacher}, and an asymptotic series with polynomial error for $d(n)$ \cite{Rho12}.

The asterisks in the second row indicate that in the case of $ss(n)$ the asymptotic formula is not explicitly stated in either of the references, but can be obtained by combining Stanley and Wright's results.  For completeness, the details and further context are discussed later in Section \ref{S:summit} of this paper.

\begin{table}[h]
\begin{center}
\begin{tabular}{|l||c|c|c|c|c|} \hline
\quad Objects \quad & \quad Functions \quad & \quad $C$ \quad & \quad $\alpha$ \quad & \quad $\beta$ \quad & \quad References \quad \\ \hline \hline

Stacks & $s(n)$  & $2^{-3} 3^{-\frac{3}{4}}$ & $\frac{5}{4}$ & $\frac{4}{3}$ & Auluck \cite{Aul51} \\ \hline

\: \dots w/ summits & $ss(n)$ & $2^{-3} {3^{-\frac{3}{4}}}^\ast$ & $\frac{5}{4}^\ast$ & $\frac{4}{3}^\ast$ & $\text{Stanley}^\ast$ \cite{Stan11}, $\text{Wright}^\ast$ \cite{Wri71}  \\ \hline

Receding stacks & $g(n)$ & $2^{-3} 3^{-\frac{1}{2}}$ & $1$ & $\frac{2}{3}$ & Auluck \cite{Aul51} \\ \hline

\: \dots w/ summits & $gs(n)$ & $2^{-2} 3^{-\frac{1}{2}}$ & $1$ & $\frac{2}{3}$ & Hardy-Ramanujan \cite{HR} \\ \hline

Shifted stacks & $h(n)$ & ? & ? & $\frac{4}{5}$ & Wright \cite{Wri72} \\ \hline

\: \dots w/ summits & $hs(n)$ & ? & ? & ? & --- \\ \hline

Strict stacks & $d(n)$ & $2^{-\frac{13}{4}} 3^{-\frac{1}{4}}$ & $\frac{3}{4}$ & $\frac{2}{3}$ & Rhoades \cite{Rho12} \\ \hline

Semi-strict stacks & $dm(n)$ & ? & ? & ? & --- \\ \hline

\end{tabular}
\end{center}
\label{Tab:stack}
\caption{Previous asymptotic formulas for stack enumeration functions, where the main asymptotic term is written in the form $C n^{-\alpha} e^{\pi \sqrt{\beta n}}$ as $n \to \infty$.  See Section \ref{S:summit} for more details on how the formula for $ss(n)$ follows from \cite{Stan11} and \cite{Wri71}.}
\end{table}

Our new results fill in the missing values and incomplete rows of Table \ref{Tab:stack}.  These missing cases are of longstanding interest, as they represent some of the earliest examples of unimodal enumeration problems, and have inspired much of the subsequent work in the area.  Indeed, shifted stacks were first introduced in Temperley's original study of boundary cases for rectangular arrays in thermal equilibrium \cite{Tem52}.  Auluck then investigated $h(n)$ at Temperley's request \cite{Aul51}, and proved that the exponential order lay in the bounded range $\frac{2}{3} \leq \beta \leq \frac{5}{6}$. The study was subsequently taken up by Wright \cite{Wri72}, who improved the asymptotic analysis by finding the exact exponent, showing that
\begin{equation}
\label{E:logh}
\log \left( h(n) \right)\sim \pi \sqrt{\frac{4}{5} n} \qquad \text{as } n \to \infty.
\end{equation}
Our first result gives precise asymptotic formula for shifted stacks and shifted stacks with summits, which is a significant improvement over \eqref{E:logh}.  We use the standard notation $\phi := \frac{1 + \sqrt{5}}{2}$ for the Golden Ratio.
\begin{theorem}
\label{T:h}
As $n \to \infty$, we have the asymptotic formulae
\begin{align*}
h(n) & \sim \frac{\phi^{-1}}{2\sqrt{2} 5^{\frac{3}{4}} n} e^{2\pi \sqrt{\frac{n}{5}}}, \\
hs(n) & \sim \phi \cdot h(n).
\end{align*}
\end{theorem}
\begin{remark*}
As pointed out to us by R. Rhoades, Ramanujan's last letter to Hardy amazingly includes the generating function for shifted stacks with summits.  Although he did not discuss the coefficients, nor their combinatorial significance, Ramanujan did provide the asymptotic expansion for the generating function (as was often the case in Ramanujan's work, he did not write his results in a strictly rigorous manner).  See \cite{Wat36} and the further discussion in Section \ref{S:shifted:summit} of this paper for additional details.  Combined with the techniques discussed in Section \ref{S:Tauberian}, Ramanujan's result could also be used to prove the asymptotic formula for $hs(n)$.  However, we provide an additional proof of this case in Section \ref{S:shifted}, as our present techniques apply equally well to both formulas in Theorem \ref{T:h}.
\end{remark*}

Our next result does not directly appear in Table \ref{Tab:stack}, but is inspired by a closer investigation of the relationship between receding stacks with summits and integer partitions.  We will review Andrews' definition of ``Frobenius symbols'' in Section \ref{S:partns:Frob}, but for the purposes of stating our result we simply denote the number of Frobenius symbols of size $n$ with no zeros in the top row by $F_\phi(n).$   We also define the notation $F_0(n)$ for the number of Frobenius symbols of size $n$ that {\it do} contain a zero in the top row.  Since there are exactly $p(n)$ Frobenius symbols of size $n$ (again, see Section \ref{S:partns:Frob} for definitions), we have
\begin{equation}
\label{E:F+F=p}
F_\phi(n) + F_0(n) = p(n).
\end{equation}

The following theorem describes the asymptotic relationship between these two functions.
\begin{theorem}
\label{T:Frob}
As $n \to \infty$,
\begin{equation*}
F_\phi(n) \sim F_0(n) \sim \frac{1}{8 \sqrt{3} n} e^{\pi \sqrt{\frac{2}{3}n}}.
\end{equation*}
\end{theorem}
\begin{remark*}
In fact, the bijections in Section \ref{S:partns} further prove that $g(n) = F_0(n)$.  This means that Theorem \ref{T:Frob} equivalently states that roughly half of all partitions (according to the uniform measure on partitions up to a given size) correspond to a Frobenius symbol that does not contain a zero in the top row, just as the ``average'' receding stack has two copies of its largest part (since $gs(n) \sim 2 g(n)$).
\end{remark*}

Our final result fills in the last row of Table \ref{Tab:stack} by giving a formula for semi-strict stacks.  Among the known results in Table \ref{Tab:stack}, the case of strict stacks was perhaps the most difficult.  The analysis for most of the other rows relied on some combination of combinatorial arguments, the transformation theory of (modular) theta functions, and/or the use of Euler-MacLaurin summation for asymptotic expansions.  However, Andrews \cite{And12} showed that $\mathscr{D}(q)$ can be expressed in terms of modular forms and Ramanujan's  famous ``mock theta functions'', which were only recently placed into the modern framework of real-analytic automorphic forms thanks to Zwegers' seminal Ph.D. thesis \cite{Zw02} (see also \cite{Zag06}).  Rhoades \cite{Rho12} then proved the asymptotic series expansion for $d(n)$ by using an extended version of the Hardy-Ramanujan Circle Method that the present authors originally developed for the asymptotic analysis of such non-modular functions in \cite{BM11}.

Indeed, it is a very challenging problem to determine the automorphic properties of a generic hypergeometric $q$-series.  Zagier's work on Nahm's conjecture in \cite{Zag06} provides a classification of modular hypergeometric $q$-series written in so-called ``Eulerian'' form, and S. Zwegers pointed out to us that this criterion shows that the shifted stack generating functions, $\mathscr{H}(q)$ and $\mathscr{H}_s(q)$, are not modular forms either.  See Section \ref{S:shifted:shifted} for additional details.

In contrast, the case of semi-strict stacks is more similar to some of the earlier rows, as $\mathscr{D}_m(q)$ can be expressed in terms of modular forms.  We prove the final missing asymptotic formula using modular transformations and Ingham's Tauberian Theorem.
\begin{theorem}
\label{T:dm}
As $n \to \infty$, we have the asymptotic formula
\begin{equation*}
dm(n) \sim \frac{1}{16n} e^{\pi \sqrt{n}}.
\end{equation*}
\end{theorem}

The remainder of the paper is structured as follows.  In Section \ref{S:partns} we define the Frobenius symbol of a partition, give a combinatorial bijection between partitions and receding stacks with summits, and finally prove Theorem \ref{T:Frob} by relating Frobenius symbols with no zeros to receding stacks.  In Section \ref{S:Tauberian} we recall Ingham's Tauberian Theorem and then use it to prove Theorem \ref{T:dm}.  We next address the case of shifted stacks, using the Constant Term Method and Saddle Point Method to prove Theorem \ref{T:h} in Section \ref{S:shifted}.  Finally, we conclude in Section \ref{S:summit} by outlining the asymptotic analysis of stacks with summits and provide further discussion of Stanley's and Wright's works.

\section*{Acknowledgments}

The authors thank Rob Rhoades and Sander Zwegers for helpful comments regarding historical works relevant to shifted stacks.

\section{Receding stacks, partitions, and Frobenius symbols}
\label{S:partns}

This section describes the combinatorial properties and asymptotic behavior of receding stacks, partitions, and Frobenius symbols.
In particular, we present a simple bijection between partitions and receding stacks with summits, and then  describe the relationship between these stacks and Frobenius symbols.
Indeed, our map essentially gives a ``conjugated'' version of the Frobenius symbol (in the sense that certain row and columns are interchanged), and it therefore comes as no surprise that our study also includes asymptotic results relevant to Andrews' study of Frobenius symbols with no zeros in the top row from \cite{And11}.

\subsection{Frobenius symbols}
\label{S:partns:Frob}

In \cite{And84F}, Andrews defined a {\it Frobenius symbol} as two equal-length sequences of decreasing nonnegative integers, usually written in two rows as
\begin{equation*}
\begin{pmatrix}
\alpha_1 & \alpha_2 & \dots & \alpha_k \\
\beta_1 & \beta_2 & \dots & \beta_k
\end{pmatrix}.
\end{equation*}
The {\it size} of such a symbol is defined to be the sum $k + \alpha_1 + \dots + \alpha_k + \beta_1 + \dots + \beta_k$.

The Frobenius symbols of size $n$ are in bijective correspondence with the partitions of size $n$.  The {\it Ferrer's diagram} of a partition $\lambda_1 + \dots + \lambda_m$ begins with a row of $\lambda_1$ blocks, which is then followed below by (left-justified) rows of $\lambda_2, \dots, \lambda_m$ blocks.  Given such a diagram, the corresponding Frobenius symbol is constructed by first setting $k$ to be the length of the main (southeast) diagnoal.  The entry $\alpha_j$ is then the length of the $j$th row that lies to the right of the diagonal, while $\beta_j$ is the length of the $j$th column that is below the diagonal.  For example, the partition $4+4+3+3+1$ corresponds to the Frobenius symbol
$$
\begin{pmatrix}
3 & 2 & 0 \\
4 & 2 & 1
\end{pmatrix}.
$$

\begin{remark*}
\label{R:Frob0}
It is straightforward to see that a Frobenius symbol with a zero in the top row corresponds to a partition in which the $k$-th largest part is exactly $k$ for some $k \geq 1$.
\end{remark*}

\subsection{A bijection between partitions and receding stacks with summits}

We now describe a bijection between partitions and receding stacks with summits that provides a combinatorial proof of the following result.  See Figure \ref{F:PartitionBij} for an example of the bijection.
\begin{proposition}
\label{P:gs=p}
For all $n \geq 0$, the number of partitions of size $n$ is equal to the number of receding stacks with summits of size $n$, so that
\begin{equation*}
gs(n) = p(n).
\end{equation*}
\end{proposition}
\begin{remark*}
This equality can also be read directly from the generating function \eqref{E:gsgen}, but the combinatorial arguments that follow provide a more refined understanding of the relationship between partitions and receding stacks with summits.
\end{remark*}
\begin{figure}[here]
\begin{center}
\includegraphics[width = 425 pt]{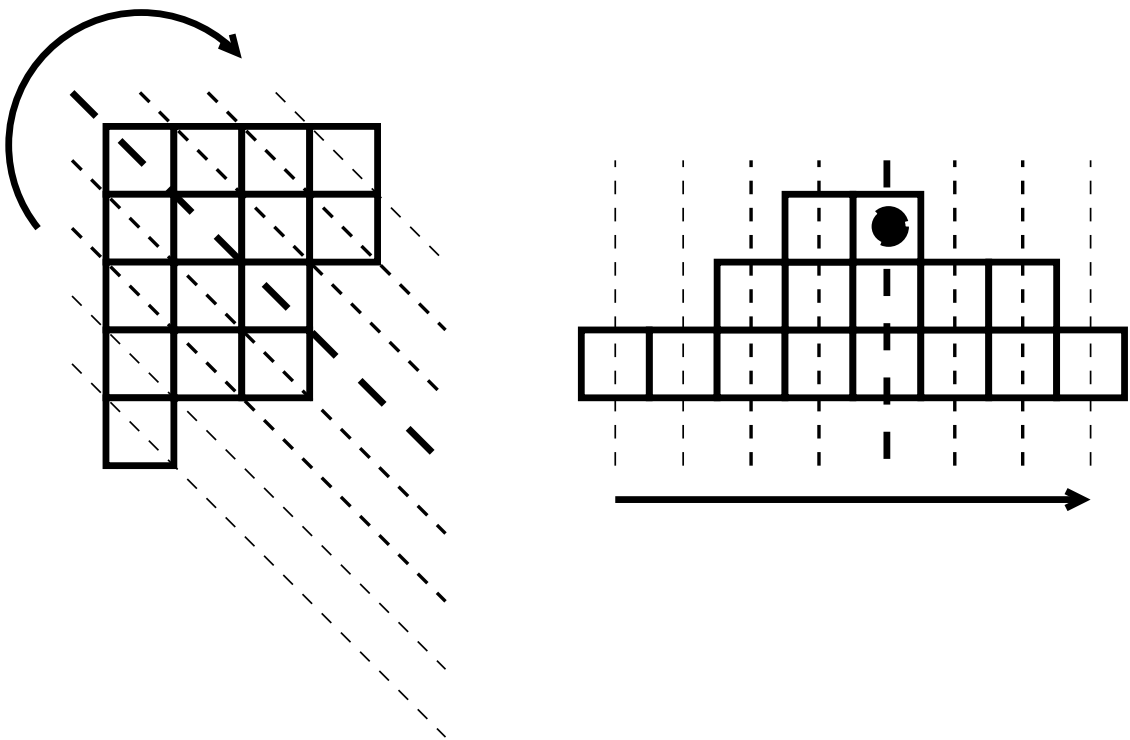}
\end{center}
\caption{\label{F:PartitionBij}
An illustration of the bijection that maps the partition $4+4+3+3+1$ to the receding stack with summit $1123\overline{3}221.$
Starting from the lower-left corner in the partition and proceeding clockwise along its boundary, the unimodal sequence is constructed by reading the blocks of the partition along diagonals.  The summit is marked on the main diagonal of the partition, which is also distinguished by the thickest dashed line in the figure.}
\end{figure}
\begin{proof}
As in the construction of the Frobenius symbol corresponding to a partition, we begin by considering separately the portion of the Ferrer's diagram to the right and below the main diagonal.  The summit $c$ of the receding stack is set to be the length of the main diagonal.  The $a_js$ are determined by reading the portion of the diagram below the main diagonal as diagonal ``rows'', beginning with the lower-left cell, which automatically gives $a_1 = 1$.  After reaching the summit $c$ along the main diagonal, the $b_js$ are then found by reading diagonal ``columns'' to the right of the main diagonal.  Proceeding along these columns, the lengths give the successive values of $b_s, b_{s-1}, \dots,$ ending with the upper-right cell, which gives $b_1 = 1$.

It is routine to verify that the receding stack conditions \eqref{E:ac}, \eqref{E:cb}, \eqref{E:arec}, and \eqref{E:brec} are always satisfied.  One needs only keep in mind that the cells in the Ferrer's diagram of a partition are fully left- and top-justified.  It is also clear that this map is a bijection, as the procedure is easily reversible, giving a well-defined inverse map from receding stacks with summits to partitions.
\end{proof}

\subsection{Frobenius symbols with no zeros}

In this section we prove Theorem \ref{T:Frob}.  As first defined in the introductory section, $F_\phi(n)$ denotes the number of Frobenius symbols of size $n$ with no zeros in the top row, and $F_0(n)$ enumerates the number of Frobenius symbols of size $n$ that do contain a zero in the top row (in particular, the final entry in the top row is the only one that can be zero).
\begin{proof}[Proof of Theorem \ref{T:Frob}]
Andrews showed in Lemma 12 (and the proof of Theorem 4) of \cite{And11} that the generating function for Frobenius symbols without zeros in the top row is
\begin{equation*}
\mathscr{F}_\phi(q) := \sum_{n \geq 0} F_\phi(n) q^n = \sum_{m \geq 0} \frac{q^{m^2 + m}}{(q)_m^2}.
\end{equation*}
The simple identity
\begin{equation*}
\sum_{m \geq 0} \frac{q^{m^2}}{(q)_m^2} = \sum_{m \geq 0} \frac{q^{m^2 + m}}{(q)_m^2} +
\left(1 + \sum_{m \geq 1} \frac{q^{m^2}}{(q)_{m-1} (q)_m}\right)
\end{equation*}
can be rewritten as
\begin{equation}
\label{E:G=F+G}
\mathscr{G}_s(q) = \mathscr{F}_\phi(q) + \mathscr{G}(q).
\end{equation}
Along with \eqref{E:gsgen} and \eqref{E:F+F=p}, this implies that
\begin{equation*}
\sum_{n \geq 0} F_0(n) q^n = \mathscr{G}(q).
\end{equation*}
We therefore have $F_0(n) = g(n)$ for all $n \geq 0$.  Furthermore, comparing coefficients in \eqref{E:G=F+G}, we also have
\begin{equation}
\label{E:gs=F+g}
gs(n) = F_\phi(n) + g(n).
\end{equation}
As alluded to in Table \ref{Tab:stack}, Auluck proved an asymptotic formula for $g(n)$; in particular, equation (15) of \cite{Aul51} states that
\begin{equation*}
g(n) \sim \frac{1}{8 \sqrt{3} n} e^{\pi \sqrt{\frac{2}{3}n}}.
\end{equation*}
Table \ref{Tab:stack} also contains Hardy and Ramanujan's famous asymptotic formula for the partition function,
\begin{equation*}
p(n)=
gs(n) \sim \frac{1}{4 \sqrt{3} n} e^{\pi \sqrt{\frac{2}{3}n}}.
\end{equation*}
Combined with \eqref{E:gs=F+g}, this completes the proof.
\end{proof}

\begin{remark*}
The discussion at the end of Section \ref{S:partns:Frob} and the proof of Proposition \ref{P:gs=p} also provide a combinatorial argument for the fact that $F_0(n) = g(n)$.  In particular, if the Frobenius symbol of a partition has a zero in its top row, then the main diagonal of its Ferrers diagram is longer than any diagonal to the right.  Using the bijection from Proposition \ref{P:gs=p}, the summit $c$ in the corresponding receding stack is therefore canonically larger than all of the $b_j$s.  The partition in Figure \ref{F:PartitionBij} is a relevant example, as the corresponding Frobenius symbol in this case was $\left(\begin{smallmatrix}
3 & 2 & 0 \\
4 & 2 & 1
\end{smallmatrix}\right).$
\end{remark*}

\section{Ingham's Tauberian Theorem and semi-strict stacks}
\label{S:Tauberian}

We begin this section by recalling a fundamental tool for analyzing the coefficients of combinatorial generating functions.  We then apply it to the case of semi-strict stacks, proving Theorem \ref{T:dm}.

\subsection{Ingham's Tauberian Theorem}
We require the following Tauberian Theorem from \cite{Ing41}, which allows us to describe the asymptotic behavior of the coefficients of a power series using the analytic nature of the series itself.
\begin{theorem}[Ingham]\label{T:Ingham}
Let $f(q)=\sum_{n\geq 0}a(n) q^n$ be a power series with weakly increasing nonnegative
coefficients and radius of convergence equal to $1$. If there are constants $A>0$,
$\lambda, \alpha\in\R$ such that
\[
f \left(e^{-\varepsilon} \right)  \sim
\lambda   \varepsilon^\alpha
\exp\left(\frac{A}{\varepsilon}\right)
\]
as $\varepsilon \to 0^+$, then
\[
a(n) \sim\frac{\lambda}{2\sqrt{\pi}}\,
\frac{A^{\frac{\alpha}{2}+\frac14}}{n^{\frac{\alpha}{2}+\frac34}}\,
\exp\left(2\sqrt{An}\right)
\]
as $n\to\infty$.
\end{theorem}

In order to apply this theorem throughout the paper, we repeatedly use the following simple observation that guarantees increasing coefficients.
\begin{lemma}
\label{L:mono}
Suppose that $A(q) = \sum_{n \geq 0} a(n) q^n$ has non-negative coefficients $a(n).$  Then the series
\begin{equation*}
\displaystyle \sum_{n \geq 0} a'(n) q^n := \frac{1}{1-q} A(q)
\end{equation*}
has monotonically increasing non-negative coefficients.
\end{lemma}
\begin{remark*}
Lemma \ref{L:mono} (and therefore Theorem \ref{T:Ingham}) are easily adapted to any series $A(q)$ whose coefficients are strictly positive after some fixed index $n'$.
\end{remark*}

\subsection{Asymptotic behavior of semi-strict stacks}

In order to apply Ingham's Tauberian Theorem to the case of semi-strict stacks, we must first understand the asymptotic behavior of the generating function $\mathscr{D}_m(q)$ (cf. \eqref{E:dmgen}).  This is best accomplished not with the expression in \eqref{E:dmgen}, but by rather using an alternative expression given by Andrews in Theorem 1, equation (1.7) of \cite{And12} (see also (1.13)).  He proved that
\begin{equation}
\label{E:Dmeta}
\mathscr{D}_m(q) = \frac{q \left(q^2; q^2\right)_\infty}{(1+q) (q)_\infty^2},
\end{equation}
which is a modular form, up to a rational expression in $q$.
If we let $q = e^{-\varepsilon}$ as in Theorem \ref{T:Ingham}, then
the modular inversion formula for Dedekind's eta-function (page 121,
Proposition 14 of \cite{Ko84}) implies that as $\varepsilon \to 0^+$
\begin{equation}
\label{E:qinfty}
\left(e^{-\varepsilon}; e^{-\varepsilon} \right)_{\infty} \sim \sqrt{\frac{2\pi}{\varepsilon}}e^{-\frac{\pi^2}{6\varepsilon}}.
\end{equation}

\begin{proof}[Proof of Theorem \ref{T:dm}]
Applying \eqref{E:qinfty} to \eqref{E:Dmeta}, we find that
\begin{equation}
\mathscr{D}_m(q) \sim \frac{1}{4 \sqrt{\pi}} \varepsilon^{\frac{1}{2}} e^{\frac{\pi^2}{4 \varepsilon}}
\end{equation}
as $\varepsilon \to 0^+.$  Furthermore, Lemma \ref{L:mono} applied to \eqref{E:Dmeta} (or to the original generating function \eqref{E:dmgen}) 
 shows that the coefficients are monotonically increasing.  Theorem \ref{T:Ingham} now directly gives the theorem statement.
\end{proof}

\section{Asymptotic behavior of shifted stacks}
\label{S:shifted}

In this section we use a variety of tools from complex analysis and number theory to prove Theorem \ref{T:h}.  In particular, we begin with the Constant Term Method, inserting an additional variable $x$ in order to identify \eqref{E:hgen} as the constant term of the product of more recognizable number-theoretic functions.  We then use Cauchy's Theorem to recover this coefficient as a contour integral; the asymptotic behavior of the integral is obtained through the Saddle Point Method.  Finally, the asymptotic behavior of $h(n)$ is recovered from an application of Ingham's Tauberian Theorem.  The asymptotic analysis of shifted stacks with summits proceeds very similarly.  It should be noted that this approach is a standard (if not widely known) technique for determining the asymptotic behavior of hypergeometric $q$-series written in Eulerian form (cf. Proposition 5 in \cite{Zag07}).  This technique can also be applied in settings where there is not an Eulerian $q$-series, such as the double summation studied in \cite{BHMV12}.

It is interesting to note that the Constant Term Method already appears in Auluck's paper \cite{Aul51}, where he used this approach to derive alternative expressions for each of the generating functions \eqref{E:sgen}, \eqref{E:ggen}, and \eqref{E:hgen}.  However, in his asymptotic analysis of shifted stacks, he then relied on somewhat oblique comparisons to the (known) asymptotic behavior of integer partitions; the same is also true of Wright's work in \cite{Wri72}.  Our main contribution is that we immediately apply analytic approximation techniques (such as modular transformations and the Saddle Point Method) after using the Constant Term Method, rather than seeking to find exact expressions for generating functions.

\subsection{Asymptotic analysis of shifted stacks and the proof of Theorem \ref{T:h}}
\label{S:shifted:shifted}

In order to study shifted stacks and prove the first asymptotic formula in Theorem \ref{T:h}, we begin by writing
\begin{align}
\label{E:CConst}
\mathscr{H} (q) &= 1+ \coeff \left[x^0\right] \left( \sum_{r\geq 0} \frac{x^{-r} q^{\frac{r^2-r}{2}}}{(q)_r} \sum_{m\geq 1} \frac{x^m q^{m}}{(q)_{m-1}} \right)  \\
 &= 1 + \coeff \left[x^0\right] \left(qx \left( -x^{-1}\right)_{\infty}  \left(xq\right)_\infty^{-1} \right), \notag
\end{align}
where we have used the following two instances of the $q$-binomial theorem (cf. Theorem 2.1 in \cite{And98}):
\begin{align*}
(x)_\infty &= \sum_{m\geq 0}\frac{(-1)^m q^{\frac{m(m-1)}{2}}x^m}{(q)_m}, \\
\frac{1}{(x)_\infty} &= \sum_{m\geq 0}\frac{x^m}{(q)_m}.
\end{align*}

In order to further rewrite \eqref{E:CConst}, we recall Jacobi's Triple Product formula (Theorem 14.6 in \cite{Apo76}), which states that
\begin{equation*}
(-x^{-1})_\infty (-xq)_\infty (q)_\infty =
- q^{-\frac{1}{8}} x^{-\frac{1}{2}} \vartheta\left(u + \frac{1}{2}; \frac{i \varepsilon}{2 \pi} \right).
\end{equation*}
Here we write $x = e^{2 \pi i u}$, $q = e^{- \varepsilon}$ with $\varepsilon>0$, and the {\it Jacobi theta function} is given by
\begin{align}
\label{E:theta}
\vartheta(u; \tau) & := \sum_{n \in \frac12+\Z}
e^{\pi i n^2\tau   + 2 \pi i n \left(u + \frac{1}{2}\right)}.
\end{align}
We  also use the {\it quantum dilogarithm} \cite{Zag07}, which is given by
\begin{align*}
\Li_2(x; q) & := -\log (x)_\infty = \sum_{m \ge 1} \frac{x^m}{m(1-q^m)}.
\end{align*}
We note that the quantum dilogarithm is a $q$-deformation of the ordinary \textit{dilogarithm}, defined for $|x|<1$ by
$$
\Li_2(x) := \sum_{m \geq 1}
\frac{x^m}{m^2}.
$$
To be more precise, for $|x|<1$ we have
$$
\lim_{\varepsilon \to 0^+} \varepsilon \Li_2\left(x; e^{-\varepsilon} \right)= \Li_2(x).
$$
Furthermore, if $|x|<1$ and $B\geq 0,$ then the Laurent expansion of the
quantum dilogarithm begins with the terms
\begin{align}
\label{TaylorLi}
\Li_2\left(e^{-B\varepsilon} x; e^{-\varepsilon}\right)
&=\sum_{m\geq 1}\frac{x^m e^{-Bm\varepsilon}}{m\left(1-e^{-\varepsilon m}\right)}
=\frac{1}{\varepsilon}\sum_{m\geq 1}\frac{x^m}{m^2}\left(1-m\varepsilon\left(B-\tfrac12\right)+
O\left(\varepsilon^2\right)\right)\\
&=\frac{1}{\varepsilon} \Li_2(x)+\left(B-\tfrac12\right)\log(1-x)+O(\varepsilon), \notag
\end{align}
where the series converges uniformly in $x$ as $\varepsilon \to 0^+$.

Using the above definitions, we can therefore rewrite \eqref{E:CConst} as
\begin{equation*}
\mathscr{H} (q)  = 1 - \frac{q^{\frac{7}{8}}}{ \left( q \right)_{\infty}} \mathscr{A} (q),
\end{equation*}
with
\begin{equation}
\label{E:AConst}
\mathscr{A} (q) := \coeff \left[x^0\right] \left( \vartheta \left( u+\frac12 ; \frac{i\varepsilon}{2\pi} \right) \exp \left( \pi i u + \text{Li}_2 \left(x^2 q^2 ; q^2 \right) \right) \right).
\end{equation}
The modular inversion formula for the Jacobi theta function (cf. Proposition 1.3 (8) in \cite{Zw02}) implies that
\begin{equation}
\label{E:thetaInv}
\vartheta \left( u+\frac12 ; \frac{i\varepsilon}{2\pi } \right)
=i \sqrt{ \frac{2 \pi}{\varepsilon}  } e^{- \frac{2\pi^2 }{\varepsilon}\left( u + \frac12 \right)^2}
\vartheta \left(\frac{2 \pi\left( u+\frac12  \right)}{i\varepsilon} ; \frac{2\pi i}{\varepsilon} \right).
\end{equation}
We use Cauchy's Theorem to recover the constant coefficient from \eqref{E:AConst} as a contour integral, and further simplify by plugging in \eqref{E:theta} and  \eqref{E:thetaInv}, obtaining
\begin{align}
\label{E:AInt}
\mathscr{A} (q) &= -\sqrt{\frac{2\pi}{\varepsilon}} \int_{[0,1] + iC}
\sum_{m\in\Z} \exp\left(-\frac{2\pi^2}{\varepsilon}  (u+m)^2 + \pi i (u+m) + \text{Li}_2 \left(x^{2}q^2 ; q^2 \right) \right) du \\
&=-\sqrt{\frac{2\pi}{\varepsilon}} \int_{\R + iC} \exp \left( -\frac{2\pi^2}{\varepsilon}  u^2 + \pi i u + \text{Li}_2 \left(x^{2}q^2 ; q^2 \right) \right) du. \notag
\end{align}
Here the contour is defined by a positive constant $C$ that will be specified later.  Note that $\text{Li}_2(x^{2}q^2; q^2)$ is (absolutely) convergent for all $u$ in the upper-half plane.

In order to continue the  analysis of the asymptotic behavior, we  require the Laurent expansion of the quantum dilogarithm. By (\ref{TaylorLi}), we have uniformly in $x$
\begin{equation}
\label{E:Li2Laurent}
\text{Li}_2 \left(x^{2}q^2; q^2\right) = \frac{1}{2\varepsilon} \text{Li}_2 \left( x^{2}\right) +\frac12 \log \left( 1- x^{2}\right) + O (\varepsilon).
\end{equation}

We are now prepared to apply the Saddle Point Method to \eqref{E:AInt}, and begin by defining the function
\begin{equation*}
f(u):= -2\pi^2 u^2 +\frac12 \text{Li}_2 \left( e^{4\pi i u} \right),
\end{equation*}
which is the leading Laurent coefficient ($\varepsilon^{-1}$-term) of the exponential argument in \eqref{E:AInt}.
Noting that
$$
\Li_2'(x) = - x^{-1} \log(1-x),
$$
we see that a critical point occurs when
\begin{equation*}
f'(u) =  2\pi i \left(2\pi i u - \log   \left( 1- e^{4\pi i u}\right) \right) = 0,
\end{equation*}
which has a unique solution in the upper half-plane given by $u = v := \frac{i}{2 \pi} \log\left(\phi\right)$, where $\phi := \frac{1 + \sqrt{5}}{2}$ is the golden ratio.  Equivalently, this point is also given by $e^{2 \pi i v} = \phi^{-1}.$  This also determines the most natural value for $C$ in \eqref{E:AInt}, namely $C = \frac{\log\left(\phi\right)}{2 \pi}$, so that the contour passes through the critical point.

At the critical point, we evaluate
\begin{equation}
\label{E:hCrit}
f(v)= \frac12 \log^2 (\phi)  +\frac12 \text{Li}_2 \left(\phi^{-2}  \right) = \frac{\pi^2}{30},
\end{equation}
using the special value (found in Section I.1 of \cite{Zag07})
\begin{equation*}
\text{Li}_2\left(\phi^{-2}\right)=\frac{\pi^2}{15}-\log^2(\phi).
\end{equation*}
We also need the value of the second derivative, so we calculate
\begin{equation*}
f''(u)
= -4\pi^2 \left( 1+\frac{2e^{4\pi i u}}{1-e^{4\pi i u}} \right).
\end{equation*}
Inserting the special value $v$,  we then find
$$
f''(v) = -4\pi^2 \left(\frac{1+\phi^{-2}}{1-\phi^{-2}} \right) = -4\pi^2 \sqrt{5}.
$$

Writing $u = v + \sqrt{\varepsilon}z$, then yields that the integrand in \eqref{E:AInt}  has the overall expansion
\begin{align*}
\exp & \left( \frac{f(u)}{\varepsilon} + \pi i u + \frac12 \log \left( 1-e^{4 \pi iu} \right)
+ O (\varepsilon) \right) \\
& = \phi^{-1} \exp \left( \frac{f(v)}{\varepsilon} + \frac{f''(v)}{2} z^2 +O \left( \varepsilon^{\frac12} \right)  \right)
= \phi^{-1} \exp\left(\frac{\pi^2}{30\varepsilon} + \frac{f''(v)z^2}{2} \right) \left( 1+O\left( \varepsilon^{\frac12} \right) \right).
\end{align*}
Inserting this into (\ref{E:AInt}) and changing variables gives  that
\begin{equation} \label{Aepsilon0}
\mathscr{A} \left(e^{-\varepsilon} \right)
=-\sqrt{2\pi} \phi^{-1} e^{\frac{\pi^2}{30\varepsilon}}
\left(1 + O\left(\varepsilon^{\frac{1}{2}}\right)\right)
 \int_{\R} e^{\frac{f''(v) z^2}{2}} dz.
\end{equation}
The (Gaussian) integral evaluates to
$$
\sqrt{\frac{2 \pi}{-f''(v)}} =\sqrt{\frac{1}{2\pi\sqrt{5}}},
$$
which implies that
\begin{equation}
\label{E:AAsymp}
\mathscr{A} \left(e^{-\varepsilon} \right)
 =-\frac{\phi^{-1}}{5^{\frac14}} e^{\frac{\pi^2}{30\varepsilon}} \left(1 + O \left( \varepsilon^{\frac12} \right) \right).
\end{equation}

Recalling the asymptotic formula \eqref{E:qinfty} and combining with \eqref{E:AAsymp}, we finally have
\begin{equation}
\label{E:CAsymp}
\mathscr{H} \left(e^{-\varepsilon} \right) \sim
  \frac{ \phi^{-1}}{\sqrt{2\pi} \, 5^{\frac14}} \varepsilon^{\frac12} e^{\frac{\pi^2}{5\varepsilon}}.
\end{equation}

In order to finish the proof of the first asymptotic formula in Theorem \ref{T:h}, we first apply Lemma \ref{L:mono} to \eqref{E:hgen} in order to conclude that $\mathscr{H}(q)$ has monotonically increasing coefficients (this can also be shown directly from the definition of shifted stacks). We can then apply Theorem \ref{T:Ingham} to \eqref{E:CAsymp}, with the constants $\lambda= \frac{\phi^{-1}}{\sqrt{2\pi} 5^{\frac14}}, \alpha=\frac12,$ and $ A = \frac{\pi^2}{5}.$  This finally gives the desired formula
$$
h(n) \sim \frac{\phi^{-1}}{2\sqrt{2} \, 5^{\frac34} n} \, e^{2\pi \sqrt{\frac{n}{5}}}.
$$

We conclude this section by discussing Zagier's classification of modular $q$-series.  If the series in \eqref{E:AConst} is re-expanded using the $q$-binomial theorem and \eqref{E:theta}, we obtain the formula
\begin{equation*}
\mathscr{H}(q) = 1 + \frac{1}{(q)_\infty} \sum_{m \geq 0} \frac{q^{(2m+1)(m+1)}}{(q^2; q^2)_m}.
\end{equation*}
This is not one of the seven examples given in Table 1 of \cite{Zag06}, and is therefore not modular.

\subsection{Shifted stacks with summits}
\label{S:shifted:summit}
We proceed as in \eqref{E:CConst}, with the only difference being some small shifts in the second sum. In particular, we express the generating function as
\begin{align*}
\mathscr{H}_s(q) &= \coeff \left[x^0\right] \left(\sum_{r\geq 0} \frac{x^{-r} q^{\frac{r(r-1)}{2}}}{(q)_r}
\sum_{m\geq 0} \frac{x^{m}q^m}{(q)_m}\right)
=1 + \coeff \left[x^0\right] \left(\left( -x^{-1}\right)_{\infty}  \left(xq\right)_\infty^{-1} \right),
\end{align*}
and thereby conclude the asymptotic formula
\begin{align}
\label{E:GsimC}
\mathscr{H}_s  \left(e^{-\varepsilon} \right)   & = 1 + \frac{q^{-\frac18}}{(q)_{\infty}} \sqrt{\frac{2\pi}{\varepsilon}} \int_{\R + iC} \exp \left(-\frac{2\pi^2 u^2}{\varepsilon} - \pi i u + \text{Li}_2 \left(x^{2}q^2;q^2 \right) \right) du \\
& \sim \phi \, \mathscr{H}  \left(e^{-\varepsilon} \right). \notag
\end{align}
The extra factor of $\phi$ arises from the fact that $e^{\pi i u}$ in the integrand of \eqref{E:AInt} has been replaced by $e^{- \pi i u}$, recalling further that the critical point is given by $e^{-2 \pi i v} = \phi$.
Arguing as before (using Lemma \ref{L:mono} on \eqref{E:hsgen} and applying Theorem \ref{T:Ingham}), we find a similar asymptotic relationship between the coefficients, concluding that
\begin{equation*}
h_s(n) \sim  \phi \cdot h (n).
\end{equation*}

\begin{remark*}
As described by Watson in \cite{Wat36} (see equation (C)), Ramanujan showed that
$$
\mathscr{H}_s(q) = (q)^{-1} G_{\frac{1}{2}}\left(q^2\right),
$$
where
\begin{equation*}
G_u(q) := \sum_{n \geq 0} \frac{q^{n(n+u)}}{(q)_n}.
\end{equation*}
The famous Rogers-Ramanujan identities \cite{RR19} state that if $u = 0$ or $1$, then
\begin{equation}
\label{E:RR}
G_u(q) = \frac{1}{\left(q^{1+r}; q^5\right)_\infty \left(q^{4-r}; q^5\right)_\infty}.
\end{equation} 
The right-side of \eqref{E:RR} is a modular function (up to a rational power of $q$) , and as such its asymptotic expansion near $q=1$ is well-understood. Under the assumption that $G_u(q)$ has an asymptotic expansion of the same shape for both integral and non-integral $u$, Ramanujan then recovered an expression equivalent to \eqref{E:GsimC}, although he did not offer any justification for this assumption.
\end{remark*}

\section{Stacks with summits}
\label{S:summit}

In this final section we briefly outline how Stanley's  and Wright's work in \cite{Stan11} and \cite{Wri71}, respectively, together imply the asymptotic formula for stacks with summits from Table \ref{Tab:stack}, namely that as $n \to \infty,$
\begin{equation}
\label{E:stackssummits}
ss(n) \sim \frac{1}{2^3 3^{\frac{3}{4}} n^{\frac{5}{4}}} e^{2\pi \sqrt{\frac{n}{3}}}.
\end{equation}
Specifically, Stanley provided a formula for the generating function of stacks with summits, using an inclusion-exclusion argument to prove Proposition 2.5.1 in \cite{Stan11}, which gives that
\begin{equation}
\label{E:SsStan}
\mathscr{S}_s(q) = \frac{1}{(q)_\infty^{2}} \big( 1 - L(q)\big).
\end{equation}
Here we have adopted Wright's notation for the ``false theta function'' (see equation (1.4) in \cite{Wri71})
\begin{equation*}
L(q) := \sum_{n \geq 1} (-1)^{n-1} q^{\frac{n(n+1)}{2}}.
\end{equation*}
Wright's own analysis of stacks in \cite{Wri71} began with the closely related expression
\begin{equation}
\label{E:SAul}
\mathscr{S}(q) = \frac{1}{(q)_\infty^{2}} L(q);
\end{equation}
which had previously appeared as equation (20) of \cite{Aul51}.  Wright's proof can then essentially be followed verbatim with $L(q)$ replaced by $1 - L(q)$, which leads to the asymptotic formula for stacks with summits.

\begin{remark*}
Although Stanley and Wright's arguments provide additional combinatorial and analytic information, we point out that \eqref{E:SsStan} also follows immediately from a specialization of Heine's $_2 \phi_1$ transformation; in particular, set $a = b = 0$ and $c = t = q$ in Corollary 2.3 of \cite{And98}.
\end{remark*}

It is perhaps worthwhile (particularly for non-expert readers) to further discuss why Wright's proof can be modified in this manner.  Wright's general argument follows the philosophy of the Circle Method, where the coefficients of a generating series are recovered by applying Cauchy's theorem, with the analysis further amplified by using asymptotic properties of the series near roots of unity.  In the case of modular forms (such as the partition function, up to a power of $q$), Hardy and Ramanujan \cite{HR} (and later Rademacher \cite{Rademacher}) used cuspidal expansions in order to derive exact formulas for the coefficients.  However, the presence of the non-modular $L(q)$ in \eqref{E:SAul} prevents a similarly powerful analysis for stacks.

Wright's modified approach instead has just one ``Major Arc'', around the exponential singularity at the cusp $q = 1$, and is useful as it can be applied to any series that can be expressed as the product of a weakly holomorphic modular form with poles at the cusps and a series with a convergent asymptotic expansion (e.g., a meromorphic Laurent series such as $L(q)$).  In particular, the terms in the asymptotic expansion of $s(n)$ can then be found from the asymptotic expansion of $L(q)$ near $q=1$, which is proven by Euler-MacLaurin summation in Lemma 1 of \cite{Wri71}.
If $q = e^{-\varepsilon}$ and $\varepsilon \to 0^+$, then the first term of Wright's expansion is stated as
\begin{equation}
\label{E:Lasymp}
L(q) = \frac{1}{2} + O(\varepsilon);
\end{equation}
in other words, $L(q) \sim 1 - L(q)$, so the asymptotic main terms for $s(n)$ and $ss(n)$ are identical.

\begin{remark*}
Alternatively, \eqref{E:stackssummits} may also be proven by applying Ingham's Tauberian Theorem, which may be viewed as an even more broadly applicable version of the Circle Method, with a correspondingly weaker conclusion that only provides the main asymptotic term.  Starting with \eqref{E:SsStan} and \eqref{E:Lasymp}, one can use \eqref{E:qinfty}, Theorem \ref{T:Ingham} and Lemma \ref{L:mono} from the present paper in order to conclude the claimed asymptotic formula for $ss(n)$.
\end{remark*}

\end{document}